\documentclass[11pt]{article}
\usepackage{amsthm}
\renewcommand{\rho}{\varrho}

\newtheorem{theorem}{Theorem}
\newtheorem{corollary}{Corollary}
\newtheorem{proposition}{Proposition}
\newtheorem{lemma}{Lemma}

\usepackage{amsmath, amssymb, mathrsfs, eucal}
\usepackage{color}

\def\tue#1\eut{}
\let\le\leqslant
\let\ge\geqslant
\def \d{\ {\rm d}}
\def\AA{\mathcal A}
\def\CC{\mathcal C}
\def\FF{\mathcal F}
\def\HH{\mathcal H}

\def\TT{\mathcal T}
\def\WW{\mathcal W}

\def\eef{\mathbb E}
\def\nnf{\mathbb N}

\def\rrf{\mathbb R}
\def\wwf{\mathbb W}

\def\one{\mathbf 1}

\long\def\tue#1\eut{}
\def\ep{\varepsilon}

\begin{document}

\title{CHARACTERISING OCONE LOCAL MARTINGALES WITH REFLECTIONS}
\author{Jean Brossard and Christophe Leuridan} 
%
%
\maketitle
\begin{abstract}
Let $M = (M_t)_{t \ge 0}$ be any continuous real-valued stochastic process such 
that $M_0=0$. Chaumont and Vostrikova proved 
that if there exists a sequence $(a_n)_{n \ge 1}$ of positive real numbers 
converging to $0$ such that $M$ satisfies the reflection principle at levels
$0$, $a_n$ and $2a_n$, for each $n \ge 1$, then $M$ is an Ocone local 
martingale. They also asked whether the reflection principle at levels
$0$ and $a_n$ only (for each $n \ge 1$) is sufficient to ensure that 
$M$ is an Ocone local martingale. 

We give a positive answer to this question, using a slightly different 
approach, which provides the following intermediate result. Let $a$ and $b$
be two positive real numbers such that $a/(a+b)$ is not dyadic. 
If $M$ satisfies the reflection principle at the level $0$ 
and at the first passage-time in $\{-a,b\}$, then $M$ is close to a local 
martingale in the following sense: $|\eef[M_{S \circ M}]| \le a+b$ 
for every stopping time $S$ in the 
canonical filtration of $\wwf = \{w \in \CC(\rrf_+,\rrf) : w(0)=0\}$ 
such that the stopped process $M_{\cdot \wedge (S \circ M)}$ is uniformly 
bounded. 
\end{abstract}

\noindent\textbf{MSC 2000:} 60G44, 60G42, 60J65.\\
Keywords: Ocone martingales, reflection principle.

\def\ge{\geqslant}

\section{Introduction}

Let $(M_t)_{t \ge 0}$ denote a continuous local martingale, defined on some 
probability space $(\Omega,\AA,P)$, such that $M_0=0$. 
Let $\FF^M$ denote its natural filtration and $\HH$ the set of all predictable 
processes with respect to $\FF^M$ with values in $\{-1,1\}$. Then for every 
$H \in \HH$, the local martingale 
$$H \cdot M = \int_0^\cdot H_s \d M_s$$
has the same quadratic variation as $M$. In particular, if $M$ is a 
Brownian motion, then $H \cdot M$ is still a Brownian motion.

A natural problem is to determine when $H \cdot M$ has the same 
law as $M$ for every $H \in \HH$.
Ocone proved in~\cite{Ocone} that a necessary and sufficient condition 
is that $M$ is a Gaussian martingale conditionally on its quadratic 
variation $\langle M \rangle$. Such processes are called 
{\it Ocone local martingales}. Various characterisations of these processes 
have been given, by Ocone himself, by Dubins, \'Emery and Yor 
in~\cite{Dubins - Emery - Yor}, by Vostrikova and Yor 
in~\cite{Vostrikova - Yor}. We refer to~\cite{Chaumont - Vostrikova} for a 
more complete presentation.  
 
The following characterisation, given by Dubins, \'Emery and Yor, is 
particularly illuminating: 
$M$ is an Ocone local martingale if and only if there exists a 
Brownian motion $\beta$ (possibly defined on a larger probability space) 
which is independent of $\langle M \rangle$ and such that 
$M_t = \beta_{\langle M \rangle_t}$ for every $t$. Loosely speaking, 
Ocone local martingales are the processes obtained by the composition 
of a Brownian motion and an independant time-change. 

Another characterisation of Ocone local martingales is based on their 
invariance with respect to reflections. For every positive real $r$, 
call $h_r$ the map from $\rrf_+$ to $\{-1,1\}$ defined by  
$$
h_r(t) = \one_{[t \le r]} - \one_{[t > r]}.
$$
Then $h_r \cdot M = 
\rho_r \circ M$, where $\rho_r$ is the {\it reflection at time} $r$. 
Let $\wwf$ denote the set of all continuous functions 
$w : \rrf_+ \to \rrf$ such that $w(0)=0$. The 
transformation $\rho_r$ maps $\wwf$ into itself and is defined by 
$$
\rho_r(w)(t) = \left\{\begin{array}{ll} w(t) & \text{ if } t \le r,\\ 2w(r)-w(t) & \text{ if } t \ge r.\end{array}\right.
$$
The functions $h_r$ are sufficient to characterise Ocone local martingales:
Theorem~A of~\cite{Ocone} states that if $h_r \cdot M$ has the same 
law as $M$ for every positive $r$, then $H \cdot M$ has the same 
law as $M$ for every $H \in \HH$. In other words, if the law of $M$ 
is invariant by the reflections at fixed times, then $M$ is an Ocone local 
martingale. Note that it is not necessary to assume that $M$ is a local 
martingale since the invariance by the reflections at fixed times implies 
that for every $t \ge s \ge 0$, the law of the increment $M_t-M_s$ is 
symmetric conditionally on $\FF^M_s$.

The celebrated reflection 
principle due to Andr\'e~\cite{Andre} shows that it may be worthwhile to 
consider reflections at first-passage times, which we now define.
For every real $a$ and $w \in \wwf$, note $T_a(w)$ the first-passage 
time of $w$ at level $a$. The reflection at time $T_a$ transforms $w$ 
into $\rho_{T_a}(w)$ where
$$
\rho_{T_a}(w)(t) = \left\{\begin{array}{ll} w(t) & \text{ if } t \le T_a(w),\\ 2a-w(t) & \text{ if } t \ge T_a(w).\end{array}\right.
$$
Note that $\rho_{T_a}(w)=w$ if $T_a(w)$ is infinite. 

Chaumont and Vostrikova recently established in \cite{Chaumont - Vostrikova} that any continuous 
process whose law is invariant by the reflections at first-passage times 
is an Ocone local martingale. Actually, their result is even stronger.

\begin{theorem}[Theorem 1 of~\cite{Chaumont - Vostrikova}]\label{t.cv}
Let $M$ be any continuous stochastic process such that $M_0=0$. 
If there exists a sequence $(a_n)_{n \ge 1}$ of positive real 
numbers converging to $0$ 
such that the law of $M$ is invariant by the reflections at 
times $T_0=0$, $T_{a_n}$ and $T_{2a_n}$, then $M$ is an Ocone local 
martingale. Moreover, if $T_{a_1} \circ M$ is almost surely finite,
then $M$ is almost surely divergent.
\end{theorem}

We note that the assumption that the law of $M$ is invariant by the reflection $\rho_0$ 
is missing in~\cite{Chaumont - Vostrikova} and that it cannot be omitted: consider 
for example the deterministic process defined by $M_t=-t$. However, if 
$\inf\{t \ge 0 : M_t>0\}$ is $0$ almost surely, the invariance by $\rho_0$ is 
a consequence of the invariance by the reflections $\rho_{T_{a_n}}$.

To prove Theorem~\ref{t.cv} above, Chaumont and Vostrikova establish 
a discrete version of the theorem and they apply it to some 
discrete approximations of $M$. The discrete version 
(Theorem~3 in~\cite{Chaumont - Vostrikova}) states that if 
$(M_n)_{n \ge 0}$ is a discrete time skip-free process 
(this means that $M_0=0$ and $M_n-M_{n-1} \in \{-1,0,1\}$ for every $n \ge 1$)
whose law is invariant by the reflections at times $T_0$, $T_1$ 
and $T_2$, then $(M_n)_{n \ge 0}$ is a discrete Ocone martingale 
(this means that $(M_n)_{n \ge 0}$ is obtained by the composition of 
a symmetric Bernoulli random walk with an independent skip-free time 
change). 

The fact that the three invariances by the reflections at times $T_0$, $T_1$, 
and $T_2$ are actually useful (two of them would not be sufficient) 
explains the surprising requirement that the law of $(M_t)_{t\ge0}$ is 
invariant by reflections at times $T_{a_n}$ and $T_{2a_n}$ in Theorem 1 
of~\cite{Chaumont - Vostrikova}. Chaumont and Vostrikova ask whether the 
assumption on $T_{2a_n}$ can be removed. Their study of the discrete 
case could lead to believe that it cannot. Yet, we give in this paper 
a positive answer to this question. Here is our main result. 

\begin{theorem}~\label{main result}
Let $M$ be any continuous stochastic process such that $M_0=0$. 
If there exists a sequence $(a_n)_{n \ge 1}$ of positive real 
numbers converging to $0$ 
such that the law of $M$ is invariant by the reflections at 
times $T_0=0$ and $T_{a_n}$, then $M$ is an Ocone local 
martingale. Moreover, if $T_{a_1} \circ M$ is almost surely finite,
then $M$ is almost surely divergent.
\end{theorem}

We provide a simpler proof of this stronger statement (the final steps 
in the approximation method of~\cite{Chaumont - Vostrikova} were rather 
technical). Let us now indicate the steps of the proof and the plan of 
the paper. 

Our proof first uses some stability properties of the set of all 
stopping times $T$ such that $\rho_T$ preserves the law of $M$.
These properties are established in section~\ref{Stability properties}. 

In section~\ref{Reflections}, we show that for any positive real numbers 
$a$ and $b$ such that $a/(a+b)$ is not dyadic, if the reflections $\rho_0$ and 
$\rho_{T_{-a} \wedge T_b}$ preserve the law of $M$, then $M$ is close 
to a local martingale in the following sense: for every stopping time 
$S$ in the canonical filtration of $\wwf$ such that 
the stopped process $M_{\cdot \wedge (S \circ M)}$ is uniformly bounded, 
$|\eef[M_{S \circ M}]| \le a+b$. To prove this, we build a nondecreasing 
sequence $(\tau_n)_{n \ge 0}$ of stopping times, increasing while finite 
($\tau_n < \tau_{n+1}$ if $\tau_n < +\infty$), 
starting with $\tau_0=0$, such that the reflections $\rho_{\tau_n}$ 
preserve the law of $M$ and such that the increments of $M$ on each interval 
$[\tau_n,\tau_{n+1}]$ are bounded by $a+b$. 

The proof that the reflections $\rho_{\tau_n}$ actually preserve the 
law of $M$ is given in section~\ref{Key fact}. 
The final step of the proof of theorem~\ref{main result} is in 
section~\ref{Proof of the main theorem}. 

To prove these results, it is more convenient to work in the canonical 
space. 
From now on, $\WW$ denotes the $\sigma$-field on $\wwf$ generated by the 
canonical projections, $X = (X_t)_{t \ge 0}$ the coordinate process on 
$(\wwf,\WW)$, and $\FF^0$ its natural filtration of the space $\wwf$ 
(without any completion). Moreover, $Q$ denotes the law of $M$ and $\eef_Q$ 
is the expectation whith respect to $Q$.

\section{Stability properties}
\label{Stability properties}

Call $\TT_Q$ the set of all stopping times $T$ of the filtration $\FF^0$ 
such that the reflection $\rho_T$ preserves $Q$. 
In this section, we establish some stability properties of $\TT_Q$. 
Let us begin with a preliminary lemma. 

\begin{lemma}~\label{Galmarino}
Let $S$ and $T$ be $\FF^0$-stopping times. If $w_1,w_2 \in \wwf$ 
coincide on $[0,T(w_1) \wedge T(w_2)]$, then
\begin{itemize} 
\item $T(w_1) = T(w_2)$;
\item either $S(w_1) = S(w_2)$ or $S(w_1) \wedge S(w_2) > T(w_1) = T(w_2)$.
\end{itemize}
Thus, the random times $S$ and $T$ are in the same order on $w_1$ as on $w_2$. 
\end{lemma}

\begin{proof}
The first point is an application of Galmarino's test 
(see~\cite{Revuz - Yor}, chapter~I, exercise~4.21). 
The second follows by the same argument, since the inequality 
$S(w_1) \wedge S(w_2) \le T(w_1) = T(w_2)$ would imply that 
$w_1,w_2 \in \wwf$ coincide on $[0,S(w_1) \wedge S(w_2)]$. 
\end{proof}

\begin{corollary}~\label{same order}
Let $T$ be an $\FF^0$-stopping time. Then 
\begin{enumerate}
\item $T \circ \rho_T = T$
\item $\rho_T$ is an involution.
\item for every $A \in \FF^0_T$, $\rho_T^{-1}(A) = A$. In particular, if $S$ 
is another stopping time, the events $\{S<T\}$, $\{S=T\}$ and $\{S>T\}$ 
are invariant by $\rho_T$. 
\end{enumerate}
\end{corollary}

\begin{proof}
The first point is a consequence of the application of the application of 
lemma~\ref{Galmarino} to the paths $w$ and $\rho_T(w)$. The secund point 
follows. The third point is another application of Galmarino's test 
(see~\cite{Revuz - Yor}, chapter~I, exercise~4.21) since $w$ and $\rho_T(w)$ 
coincide on $[0,T(w)]$. 
\end{proof}

The next lemma states that $\TT_Q$ is stable by the optional mixtures.

\begin{lemma}~\label{mixture} 
Let $(S_n)$ be a (finite or infinite) sequence of $\FF^0$-stopping times 
and $(A_n)$ a measurable partition of $(\wwf,\WW)$ such that 
$A_n \in \FF_{S_n}$ for every $n$. Then 
$$T := \sum_n S_n \one_{A_n}$$
is an $\FF^0$-stopping time. If $S_n \in \TT_Q$ for every $n$, then 
$T \in \TT_Q$.
\end{lemma}

\begin{proof}
Note that $T$ is an $\FF^0$-stopping time since for every $t \in \rrf_+$, 
\begin{eqnarray*}
\{T \le t\} = \bigcup_n (A_n \cap \{S_n \le t\}) \in \FF_t.
\end{eqnarray*}
Fix any bounded measurable function $\phi$ from $\wwf$ to $\rrf$. 
Since for each $n$, the event $A_n$ and the probability $Q$ are invariant 
by $\rho_{S_n}$, one has
\begin{eqnarray*}
\eef_Q[\phi \circ \rho_T] 
&=& \sum_n \eef_Q[(\phi \circ \rho_{S_n})\one_{A_n}] \\
&=& \sum_n \eef_Q[(\phi\one_{A_n}) \circ \rho_{S_n}] \\
&=& \sum_n \eef_Q[\phi\one_{A_n}] \\
&=& \eef_Q[\phi].
\end{eqnarray*}
Hence $\rho_T$ preserves $Q$. 

\end{proof}
 
\begin{corollary}~\label{min and max}
For every $S$ and $T$ in $\TT_Q$, $S \wedge T$ and $S \vee T$ are in $\TT_Q$.
\end{corollary}

\begin{proof}
As the events $\{S<T\}$, $\{S=T\}$ and $\{S>T\}$ belong to 
$\FF_S \cap \FF_T$, the result is a direct application of lemma~\ref{mixture}.
\end{proof}

The following lemmas will be used to prove a subtler result: if 
$S$ and $T$ are in $\TT_Q$, then $S \circ \rho_T$ is  in $\TT_Q$. 

\begin{lemma}~\label{preservation of the filtration}
Let $S$ and $T$ be $\FF^0$-stopping times. Then the following holds.
\begin{itemize} 
\item For every $t \ge 0$, $\rho_T^{-1}(\FF_t) = \FF_t$.
\item $S \circ \rho_T$ is an $\FF^0$-stopping time.
\end{itemize}
\end{lemma}

\begin{proof}
Fix $t \ge 0$. Then $\rho_T^{-1}(\FF_t)$ is the $\sigma$-field 
generated by the random variables $X_s \circ \rho_T$ for $s \in [0,t]$, and 
the equality
$$X_s \circ \rho_T = (2X_T-X_s)\one_{[T \le s]} + X_s\one_{[T > s]}$$
shows that these random variables are measurable for $\FF_t$.
Thus $\rho_T^{-1}(\FF_t) \subset \FF_t$. Since $\rho_T$ is an involution, 
the reverse inclusion follows, which proves the first statement.

For each $t \ge 0$, 
$\{S \circ \rho_T \le t\} = \rho_T^{-1}(\{S \le t\}) \in \FF_t$,
which proves the second statement.
\end{proof}

\begin{lemma}~\label{formulas}
Let $S$ and $T$ be $\FF^0$-stopping times and $w \in \wwf$.\\
If $S(w) \le T(w)$, then $S(\rho_T (w)) = S(w)$ and 
$\rho_{S \circ \rho_T} (w) = \rho_S(w)$.\\
If $S(w) \ge T(w)$, then $T(\rho_S(\rho_T (w))) = T(w)$ 
and $\rho_{S \circ \rho_T} (w) = \rho_T(\rho_S(\rho_T (w)))$.
\end{lemma}

\begin{proof}
If $S(w) \le T(w)$, then $w$ and $\rho_T (w)$ coincide on $[0,S(w)]$, 
thus $S(\rho_T (w)) = S(w)$ and $\rho_{S \circ \rho_T} (w) = \rho_S(w)$.

If $S(w) \ge T(w)$, then $S(\rho_T (w)) \ge T(\rho_T (w)) = T(w)$ by 
corollary~\ref{same order}, thus
$\rho_S(\rho_T (w))$, $\rho_T (w)$ and $w$ 
coincide on $[0,T(w)]$, thus $T(\rho_S(\rho_T (w))) = T(w)$. But, to get 
$\rho_T \circ \rho_S \circ \rho_T (w)$ from $w$, one must successively: 
\begin{itemize}
\item multiply by $-1$ the increments after $T(w)$; 
\item multiply by $-1$ the increments after $S(\rho_T (w))$;
\item multiply by $-1$ the increments after $T(\rho_S(\rho_T (w)))$.
\end{itemize}
Since $T(\rho_S(\rho_T (w))) = T(w)$, one gets
$\rho_{S \circ \rho_T} (w) = \rho_T \circ \rho_S \circ \rho_T (w)$.
\end{proof}

\begin{lemma}~\label{composition with reflections}
For every $S$ and $T$ in $\TT_Q$, $S \circ \rho_T$ belongs to $\TT_Q$.
\end{lemma}

\begin{proof}
By lemma~\ref{formulas} and corollary~\ref{same order}, 
one has, for every $B \in \WW$, 
\begin{eqnarray*}
Q[\rho_{S \circ \rho_T}^{-1}(B)] 
&=& Q[\rho_{S \circ \rho_T}^{-1}(B)\ ;\ S \le T] 
+Q[\rho_{S \circ \rho_T}^{-1}(B)\ ;\ S > T]\\ 
&=& Q[\rho_{S}^{-1}(B \cap \{S \le T\})] 
+Q[(\rho_T \circ \rho_S \circ \rho_T)^{-1}(B \cap \{S > T\})]\\ 
&=& Q[B \cap \{S \le T\}] + Q[B \cap \{S > T\}] = Q[B].
\end{eqnarray*}
Thus $S \circ \rho_T$ belongs to $\TT_Q$.
\end{proof}

Here is a simple application of our last lemmas.

\begin{corollary}~\label{composition with the reflection at $0$}
For every $a \in \rrf$, $T_{-a} = T_a \circ \rho_0$ and 
$\rho_{T_{-a}} = \rho_0 \circ \rho_{T_a} \circ \rho_0$.
Thus, if $0\in \TT_Q$ and $T_a\in \TT_Q$, then $T_{-a} \in \TT_Q$.
\end{corollary}

\begin{proof}
The first equality is obvious and 
the second equality follows from lemma~\ref{formulas}. 
One can deduce the last point either from the first equality by
lemma~\ref{composition with reflections} or directly 
from the second equality.
\end{proof}

\section{Reflections at $0$ and at the hitting time of $\{-a,b\}$}
\label{Reflections}

We keep the notations of the previous section and we fix two positive real 
numbers $a,b$ such that $a/(a+b)$ is not dyadic. 
Note that $T = T_{-a} \wedge T_b$ is the hitting time of $\{-a,b\}$. 
This section is devoted to the proof of the following result. 

\begin{proposition}~\label{bound}
Let $Q$ be a probabiliy measure on $(\wwf,\WW)$.
If $0\in\TT_Q$ and $T\in\TT_Q$, 
then, for every finite stopping time $S$ in the canonical filtration of $\wwf$ 
such that the stopped process $X_{\cdot \wedge S}$ is uniformly bounded, one has
$$|\eef_Q[X_S]| \le a+b.$$
\end{proposition}

Note that the process $X$ may not be a local martingale.
The law of any process which stops when its absolute value hits $\min(a,b)$
fulfills the assumptions provided it is invariant by $T_0$. 
  
The requirement that $a/(a+b)$ is not dyadic may seem surprising, 
and one could think that it is just a technicality provided by the method
used to prove the result. In fact, proposition~\ref{bound} 
becomes false if this assumption is removed. A simple counterexample is 
given by the continuous stochastic process $(M_t)_{t \ge 0}$ defined by 
$$
M_t = \left\{\begin{array}{cl} t\xi &\text{ if } t \le 1,\\ \xi+(t-1)\eta & \text{ if } t > 1,\end{array}\right.
$$
where $\xi$ and $\eta$ are independent symmetric Bernoulli random variables.
Indeed, the law $Q$ of $M$ is invariant by reflections at times $0$ and 
$T_{-1} \wedge T_1$ since $T_{-1} \wedge T_1 = 1$ $Q$-almost surely. 
Yet, for every $c > 1$, 
the random variable $X_{T_{-2} \wedge T_c}$ is uniform on $\{-2,c\}$ and its 
expectation $(c-2)/2$ can be made as large as one wants.
 
The proof of proposition~\ref{bound} uses an increasing sequence of 
stopping times defined as follows. 
Call $D$ the set of $c \in ]-a,b[$ such that $(c+a)/(b+a)$ is not dyadic.
For every $x \in D$, set 
$$f(x)= \left\{\begin{array}{cl}  
2x+a &\text{ if } x < (b-a)/2,\\ 
2x-b &\text{ if } x > (b-a)/2. 
\end{array}\right.$$
This defines a map $f$ from $D$ to $D$. Conjugating $f$ by the affine 
map which sends $-a$ on $0$ and $b$ on $1$ gives the classical map 
$x \mapsto 2x \mod 1$ restricted to the non-dyadic elements of $]0,1[$.

By hypothesis, $0 \in D$, so one can define an infinite sequence 
$(c_n)_{n \ge 0}$ of elements of $D$ by $c_0=0$, and $c_n=f(c_{n-1})$ 
for $n \ge 1$. By definition, $c_{n-1}$ is the middle point
of the subinterval $[c_n,d_n]$ of $[-a,b]$, where 
$$
d_n= \left\{\begin{array}{cl} -a & \text{ if } c_{n-1} < (b-a)/2,\\
 b & \text{ if } c_{n-1} > (b-a)/2.\end{array}\right.
 $$
Note that $|c_n-c_{n-1}| = d(c_{n-1},\{-a,b\})$.

We define a sequence $(\tau_n)_{n \ge 0}$ of stopping times on $\wwf$ by 
setting $\tau_0=0$, and for every $n \ge 1$,
\begin{eqnarray*}
\tau_n(w) 
= \inf\{t \ge \tau_{n-1}(w) : |w(t)-w(\tau_{n-1}(w))| = |c_n-c_{n-1}|\}.
\end{eqnarray*}
Note that $c_n \ne c_{n-1}$ for every $n \ge 1$, hence the 
sequence $(\tau_n(w))_{n \ge 0}$ is increasing.
Moreover, since $(|c_n-c_{n-1}|)_{n \ge 1}$ does not converge to $0$, 
the continuity of $w$ forces the sequence $(\tau_n(w))_{n \ge 0}$
to be unbounded. By convention, we set $\tau_\infty=+\infty$.

Note that if $a=-1$ and $b=2$, then $c_n=0$ for every even $n$ and 
$c_n=1$ for every odd $n$ and the sequence $(\tau_n)_{n \ge 0}$ is 
similar to the sequences used in~\cite{Chaumont - Vostrikova}.

The proof of proposition~\ref{bound} relies on the following key statement.

\begin{proposition}~\label{key}
If $0\in\TT_Q$ and $T\in\TT_Q$, then
$\tau_n\in\TT_Q$ for every $n \ge 0$.
\end{proposition}

This statement, that will be proved in the next section, has a remarkable 
consequence. 
   
\begin{corollary}
If $0\in\TT_Q$ and $T\in\TT_Q$,
then the sequence $(Y_n)_{n \ge 0}$ of random variables defined 
on the probability space $(\wwf,\WW,Q)$ by
$$Y_n(w) = X_{\tau_{D_n}}(w) \text{ where } 
D_n(w) = \max\{k \le n : \tau_k(w) < +\infty\}.$$
is a martingale in the filtration 
$(\FF^0_{\tau_n})_{n \ge 0}$.
\end{corollary}

\begin{proof}
Fix $n \ge 0$. The equality
$$Y_n(w) = 
\sum_{k=0}^{n-1} \one_{[\tau_{k}(w) < +\infty\ ;\ \tau_{k+1}(w) = +\infty]}
X_{\tau_k}(w) + \one_{[\tau_{n}(w) < +\infty]} X_{\tau_n}(w),$$
shows that $Y_n$ is measurable for $\FF^0_{\tau_n}$.
Moreover, from the equality
$$Y_{n+1}(w) - Y_n(w) 
= (X_{\tau_{n+1}}(w)-X_{\tau_n}(w)) \one_{[\tau_{n+1}(w) < +\infty]},$$
we deduce that $(Y_{n+1} - Y_n) \circ \rho_{\tau_n} = - (Y_{n+1} - Y_n)$. 
Take $A \in \FF^0_{\tau_n}$. Then $\rho_{\tau_n}^{-1}(A)=A$, since 
every $w \in \wwf$ coincide with $\rho_{\tau_n}(w)$ on $[0,\tau_{n}(w)]$. 
Since the reflection $\rho_{\tau_n}$ preserves $Q$, we get
$$\eef_Q[(Y_{n+1} - Y_n) \one_A] = 
\eef_Q\big[\big((Y_{n+1} - Y_n) \one_A\big)\circ \rho_{\tau_n} \big] 
= - \eef_Q[(Y_{n+1} - Y_n) \one_A],$$
which shows that $\eef_Q[Y_{n+1} - Y_n|\FF^0_{\tau_n}] = 0$. 
\end{proof}

We are now ready to prove proposition~\ref{bound}.

\begin{proof}
Fix $C \in \rrf_+$ such that $|X_{t \wedge S}(w)| \le C$ for every 
$t \in \rrf_+$ and $w \in \wwf$.
For each $w \in \wwf$, set $N(w) = \inf\{n \ge 1 : \tau_n(w) \ge S(w)\}$. 
Since $S(w)$ is finite and $\tau_n(w)$ is unbounded as $n \to +\infty$, 
$N(w)$ is finite. 

For every $n \ge 0$, $\{N \le n\} = \{\tau_n \ge S\} \in \FF^0_{\tau_n}$. 
Thus $N$ is a stopping time and $(Y_{n \wedge N})_{n \ge 0}$ is a martingale 
in the filtration $(\FF^0_{\tau_n})_{n \ge 0}$. Note that:
\begin{itemize}
\item for all $n < N(w)$, one has $\tau_{D_n}(w) < S(w)$ 
hence $|Y_n(w)| = |X_{D_n(w)}| \le C$,
\item and $|Y_N(w)| \le |Y_{N-1}(w)| + |Y_N(w)-Y_{N-1}(w)| \le C + (a+b)/2$.
\end{itemize}
This shows that the martingale $(Y_{n \wedge N})_{n \ge 0}$ is uniformly 
bounded, hence it converges in $L^1(Q)$ to $Y_N$ and 
$\eef_Q[Y_N]=\eef_Q[Y_0]=0$.

Note that $\tau_{N-1} < S < +\infty$, hence $Y_N=X_{\tau_N}$ or 
$Y_N=X_{\tau_{N-1}}$.
The inequalities $\tau_{N-1} < S \le \tau_N$ and the fact that the increments 
of $X$ are bounded by $a+b$ on each interval $[\tau_{n-1},\tau_n[$ yield 
$|X_S-Y_N| \le a+b$. This completes the proof. 
\end{proof}

\section{Proof of proposition~\ref{key}}
\label{Key fact}
 
We keep the notations of the previous section, and we introduce for every 
$n \ge 1$, 
$$\ep_n = \frac{Y_n-Y_{n-1}}{c_n-c_{n-1}} 
= \frac{X_{\tau_n}-X_{\tau_{n-1}}}{c_n-c_{n-1}} \one_{[\tau_n < +\infty]}.$$
For every $e=(e_n)_{n \ge 1} \in \{-1,0,1\}^\infty$, set 
$$m_0(e) = \inf\{n \ge 1 : e_n=0\},\quad
m(e) = \inf\{n \ge 1 : e_n=-1\}.$$
Call $\Sigma$ the set of all sequences 
$e=(e_n)_{n \ge 1} \in \{-1,0,1\}^\infty$ such that 
$e_n=0$ for all $n \ge m_0(e)$.  
Then $\ep=(\ep_n)_{n \ge 1}$ can be seen as a random variable 
with values in $\Sigma$. 

The first key point is that $T$ is always one of the times $(\tau_n)_{n \ge 1}$.

\begin{lemma}~\label{identity}
One has $T = \tau_{m \circ \ep}$ (remind the convention $\tau_\infty=+\infty$). 
Thus, for every $n \ge 1$, 
$\{m \circ \ep = n\} = \{T = \tau_n < +\infty\}$.
\end{lemma}

\begin{proof}
Fix $w \in \wwf$ and set $m=m(\ep(w))$ and $m_0=m_0(\ep(w))$. 

For every $n \ge 1$, by definition of $\tau_n$ and $\ep_n$, one has 
$$\begin{array}{rl}
\ep_n(w) = \pm 1 & \text{ if } \tau_n(w) < +\infty,\\
\ep_n(w) = 0   & \text{ if } \tau_n(w) = +\infty.
\end{array}$$
In particular, $\tau_n(w) = +\infty$ for every $n \ge m_0$ 
since the sequence $(\tau_n)_{n \ge 0}$ is non decreasing.
Thus, whether $m \le m_0$ or $m \ge m_0$, one has 
$\tau_{m \wedge m_0}(w)=\tau_{m}(w)$.

For every $k < m \wedge m_0$, $\ep_k(w)=1$ hence 
$w(\tau_k(w))-w(\tau_{k-1}(w)) = c_k-c_{k-1}$. 
A recursion then gives $w(\tau_k(w))=c_k \in ]-a,b[$. 
Moreover, for $\tau_k(w) \le t < \tau_{k+1}(w)$, 
$$|w(t)-c_k| = |w(t)-w(\tau_k(w))| < |c_{k+1}-c_k| = d(c_k,\{-a,b\}).$$
Hence for every $t \in [0,\tau_m(w)[$, $w(t) \notin \{-a,b\}$. This proves 
that $T(w) \ge \tau_m(w)$. 

If $m$ is infinite, then $T(w)$ is infinite.

If $m$ is finite, the equality
$$w(\tau_m(w))-w(\tau_{m-1}(w)) = -(c_m-c_{m-1}) = d_m-c_{m-1}$$
implies $w(\tau_m(w))=d_m \in \{-a,b\}$, hence $T(w) = \tau_m(w)$. 

The proof of the first statement is complete. Since the 
sequence $(\tau_n(w))_{n \ge 0}$ is increasing and unbounded,
the second statement follows.
\end{proof}

We can now describe the effect of the reflection 
$\rho_T$ on the sequence $\ep=(\ep_n)_{n \ge 1}$. 
For every $e=(e_n)_{n \ge 1} \in \Sigma$, define 
$r(e)=(f_n)_{n \ge 1} \in \Sigma$ by
$$f_n= \left\{\begin{array}{cl}
e_n & \text{ if } n \le m(e),\\
-e_n & \text{ if } n > m(e).
\end{array}\right.$$ 
Set $g(e)=r(-e)$ and $\gamma = \rho_T \circ \rho_0$.
Note that $g$ and $\gamma$ are bijective maps.

\begin{corollary}
With the notation above, the following properties hold.
\begin{enumerate}
\item The reflections $\rho_0$, $\rho_{T}$ and their 
composition $\gamma=\rho_T \circ \rho_0$ 
preserve the stopping times $\tau_n$.
\item One has $\ep \circ \rho_0 = -\ep$, 
$\ep \circ \rho_{T} = r \circ \ep$ and 
$\ep \circ \gamma = g \circ \ep$.
\end{enumerate}
\end{corollary}

\begin{proof}
Let $w \in \wwf$. The trajectories $w$ and $\rho_{T}(w)$ have the same 
increments on 
$[0,T(w)]$ and have opposite increments on $[T(w),+\infty[$. Since 
$T(w) = \tau_{m \circ \ep}(w)$, the results on $\rho_T$ follow immediatly. 
The other statements are obvious.
\end{proof}

For $n \in \nnf$, note $\one_n = (1,\ldots,1) \in \{-1,1\}^n$. 
For $(e_1,\ldots,e_n) \in \{-1,1\}^n$ and $\sigma \in \Sigma$,
note $(e_1,\ldots,e_n,\sigma) \in \Sigma$ the sequence obtained by 
concatenation. The next formula will play the same role as lemma~1 
of~\cite{ Chaumont - Vostrikova}. 

\begin{lemma}~\label{explicit formula}
Let $N = a_0+2a_1+\cdots+2^{n-1}a_{n-1}$ be a natural integer written 
in base $2$ with $n$ digits (the digit $a_{n-1}$ may be $0$). Then for 
every $\sigma \in \Sigma$,
$$g^N(\one_n,\sigma) = ((-1)^{a_0},\ldots,(-1)^{a_{n-1}},\sigma).$$
Moreover, if $n \ge 1$, 
$$g^{N-2^{n-1}}(\one_{n-1},-1,\sigma) 
= ((-1)^{a_0},\ldots,(-1)^{a_{n-1}},\sigma).$$
\end{lemma}

\begin{proof}
The first formula will be proved by induction on the number of digits. 
If $n=0$, then $N=0$ and the formula is obvious.

Assume the formula holds for all integers written with $n$ digits. 
Let $N = a_0+2a_1+\cdots+2^{n}a_{n}$ be an integer written with $n+1$ 
digits. 

If $a_n=0$, then it suffices to write $N$ with $n$ digits and 
to apply the induction hypothesis to the sequence $(1,\sigma)$. 

If $a_n=1$, let us apply the induction hypothesis to the integer 
$2^n-1 = 1+2+\cdots+2^{n-1}$ and to the sequence $(1,\sigma)$. We get  
$$g^{2^n-1}(\one_{n+1},\sigma) = (-\one_n,1,\sigma).$$
Applying $g$ once more yields 
$$g^{2^n}(\one_{n+1},\sigma) = (\one_n,-1,\sigma).$$
Applying the induction hypothesis to the integer 
$$
N-2^n = a_0+2a_1+\cdots+2^{n-1}a_{n-1}
$$ and to the sequence $(-1,\sigma)$
yields 
$$g^{N}(\one_{n+1},\sigma) = ((-1)^{a_0},\ldots,(-1)^{a_{n-1}},-1,\sigma),$$
which achieves the proof of the first formula. 

In particular, if $n \ge 1$, 
$g^{2^{n-1}}(\one_n,\sigma) = (\one_{n-1},-1,\sigma)$, hence
$$g^{-2^{n-1}}(\one_{n-1},-1,\sigma) = (\one_n,\sigma).$$ 
The second formula follows.
\end{proof}

Introduce $\Sigma_n \subset \{-1,0,1\}^n$ the subset of $n$-uples such that 
each 
component after a $0$ is $0$. Define the map $g$ from $\Sigma_n$ to itself 
as before. 

\begin{corollary}
For every $n \ge 1$ and $e=(e_1,\ldots,e_n) \in \Sigma_n$, there exists 
an integer $M(e)$ such that the event 
$A_e = \{(\ep_1,\ldots,\ep_n) = (e_1,\ldots,e_n)\}$ belongs to 
$\FF^0_{T \circ \gamma^{M(e)}}$ and $\tau_n = T \circ \gamma^{M(e)}$ 
on $A_e$. 
\end{corollary}


\begin{proof}
Set $(e_1,\ldots,e_n)=((-1)^{a_0},\ldots,(-1)^{a_{d-1}},0,\ldots,0)$
with $0\le d \le n$ and $a_0,\ldots,a_{d-1} \in \{0,1\}$.

If $d=n$, set $M(e) = 2^{n-1}-a_0-\cdots-2^{n-1}a_{n-1}$. 
Then by lemmas~\ref{explicit formula} and~\ref{identity}, 
\begin{eqnarray*}
A_e
&=& \{g^{M(e)} \circ (\ep_1,\ldots,\ep_n) = (\one_{n-1},-1)\}\\
&=& \{(\ep_1,\ldots,\ep_n) \circ \gamma^{M(e)} = (\one_{n-1},-1)\}\\
&=& \{m \circ \ep  \circ \gamma^{M(e)} = n\}\\
&=& \{T \circ \gamma^{M(e)} = \tau_n \circ \gamma^{M(e)} < +\infty\}
\end{eqnarray*}
Thus $A_e \in \FF^0_{T \circ \gamma^{M(e)}}$, and 
$\tau_n = \tau_n \circ \gamma^{M(e)} = T \circ \gamma^{M(e)}$ on $A_e$.

If $d \le n-1$, set $M(e) = -a_0-\cdots-2^{d-1}a_{d-1}$. Then by 
lemmas~\ref{explicit formula} and~\ref{identity}, 
\begin{eqnarray*}
A_e
&=& \{g^{M(e)} \circ (\ep_1,\ldots,\ep_n) = (\one_{d},0,\ldots,0)\}\\
&=& \{(\ep_1,\ldots,\ep_n) \circ \gamma^{M(e)} = (\one_{d},0,\ldots,0)\}\\
&=& \{m_0 \circ \ep  \circ \gamma^{M(e)} = d+1\ ;
\ m \circ \ep  \circ \gamma^{M(e)} = +\infty \}\\
&=& \{\tau_d \circ \gamma^{M(e)} < +\infty\ ;
\ T \circ \gamma^{M(e)} = \tau_{d+1} \circ \gamma^{M(e)} = +\infty\}
\end{eqnarray*}
Thus $A_e \in \FF^0_{T \circ \gamma^{M(e)}}$, and 
$\tau_n = \tau_n \circ \gamma^{M(e)} = +\infty = T \circ \gamma^{M(e)}$ on $A_e$.
\end{proof}

The last corollary and the stability properties given in 
lemmas~\ref{composition with reflections} and ~\ref{mixture} 
show that if $0 \in \TT_Q$ and $T \in \TT_Q$, then $\tau_n \in \TT_Q$
for all $n \ge 0$ (recall that $\tau_0 = 0$). 
This ends the proof of proposition~\ref{key}.

\section{Proof of the main theorem}
\label{Proof of the main theorem}

Let us now prove theorem~\ref{main result}. 

\begin{proof}
Call $Q$ the law of $M$ as before. The first step of the proof is 
the observation that for every integer $n \ge 1$, $T_{-a_n} \in \TT_Q$ 
by corollary~\ref{composition with the reflection at $0$}.
Hence for all integers $m,n \ge 1$, $T_{-a_n} \wedge T_{a_m} \in \TT_Q$ 
by corollary~\ref{min and max}. Lemma~\ref{non dyadic}, which will be 
stated and proved below, ensures that the ratios $a_n/(a_n+a_m)$ are not dyadic 
for arbitrarily large $m$ and $n$. For such $m$ and $n$, 
proposition~\ref{bound} applies and yields $\eef_Q[X_S] \le a_n+a_m$ 
for every finite stopping time $S$ (in the canonical filtration $\WW$)
such that the stopped process $X_{\cdot \wedge S}$ is uniformly bounded.
Since $(a_n)_{n \ge 1}$ converges to $0$, this proves that $\eef_Q[X_S]=0$,
hence $X$ is a local martingale under $Q$. 

The next arguments are the same as in~\cite{Chaumont - Vostrikova}
and we now summarize them. 

$Q$-almost surely, the process $X$ admits a quadratic variation 
$\langle X \rangle$ (defined as a limit in probability of sums of 
squared increments), which is preserved by the reflections 
$\rho_0$ and $\rho_{T_{a_n}}$. 
Consider a regular version of the conditional law of $X$ with respect 
to $\langle X \rangle$. For any continuous non-decreasing function 
$f : \rrf_+ \to \rrf_+$ such that $f(0)=0$, call $Q_f$ the law of $X$ 
conditionally on $\langle X \rangle = f$. 
Then for almost every $f$ (for the law of $\langle X \rangle$), 
the probability $Q_f$ is 
invariant by the reflections $\rho_0$ and $\rho_{T_{a_n}}$. 

By the part of the theorem which is already proven, $X$ is a local 
martingale under $Q_f$. But $\langle X \rangle = f$ almost surely 
under $Q_f$. Calling $\phi$ the right-continuous inverse of $f$, 
one gets that the process $B = (X_{\phi(s)})_{0 \le s < f(+\infty)}$ 
is a Brownian motion with lifetime $f(+\infty)$. 

Consider, in some suitable enlargement of the probability space $(\wwf,\WW,Q)$, 
a Brownian motion $W$, independent of $X$. For almost every $f$, 
the Brownian motion $W$ is still independent of $X$ under $Q_f$.
Since the local martingale $X$ converges $Q$-almost surely to a 
random variable $X_\infty$ on the event 
$\{\langle X \rangle_\infty < +\infty\}$, one gets a Brownian
motion $B$ defined on the whole interval $[0,+\infty[$ and independent 
of $\langle X \rangle$ by setting 
$$B_s = X_\infty + W_{s-\langle X \rangle_\infty} \text{ on the event } 
\{\langle X \rangle_\infty \le s\}.$$
Since $X_t=B_{\langle X \rangle_t}$ almost surely for all $t \ge 0$, 
this shows that $X$ is an Ocone local martingale under $Q$.

Assume now that $\langle X \rangle_\infty$ is finite with positive 
probability. Then for some $s \in \rrf_+$, $\langle X \rangle_\infty \le s$ 
with positive probability. But with positive probability, $B$ does not 
visit $a_1$ before time $s$. By independence of $B$ and $\langle X \rangle$,   
$$Q[T_{a_1} = +\infty] \ge Q[T_{a_1} \circ B > s]\ 
Q[\langle X \rangle_\infty \le s] > 0.$$
This shows that if $T_{a_1}$ is finite $Q$-almost surely,
then $\langle X \rangle_\infty$ is infinite $Q$-almost surely,
hence $X$ is almost surely divergent.
\end{proof}

Note that the proof of the last statement (if $T_{a_1}$ is finite 
$Q$-almost surely, then $X$ is almost surely divergent) given in the 
discrete case by Chaumont and Vostrikova 
(lemma 2 of~\cite{Chaumont - Vostrikova}) 
is not correct because they prove the implication 
$$T_a(M) \vee T_{-a}(M) < +\infty \text{ a.s. } \Longrightarrow 
T_{a+2}(M) \wedge T_{-a-2}(M) < +\infty \text{ a.s.,}$$ 
which is not sufficient to perform an induction. 
Yet, the same arguments 
that Chaumont and Vostrikova used to prove their lemma 1 are sufficient to 
prove their lemma 2. Our lemma~\ref{explicit formula} generalises these 
arguments, and the case in which some stopping time $T_a$ is infinite is 
covered by the possibility for the sequence of signs $\sigma \in \Sigma$ to  
be eventually $0$.

\begin{lemma}~\label{non dyadic}
If $c>b>a>0$, then at least one of the three following ratios $a/(a+b)$, 
$b/(b+c)$ and $a/(a+c)$ is not dyadic.
\end{lemma}

\begin{proof}
The three ratios above belong to $]0,1/2[$. Assume that they are dyadic. 
Then  
$$\frac{a}{a+b} = \frac{i}{2^p},\quad \frac{b}{b+c} = \frac{j}{2^q},
\quad \frac{a}{a+c} = \frac{k}{2^r},$$  
where $i$, $j$ and $k$ are odd positive integers and $p$, $q$ and $r$ are integers greater 
or equal to $2$. Thus
$$\frac{2^r-k}{k} = \frac{c}{a} = \frac{b}{a} \times \frac{c}{b} 
= \frac{2^p-i}{i} \times \frac{2^q-j}{j},$$
$$ij(2^r-k) = k(2^p-i)(2^q-j),$$
$$2^rij + 2^qik + 2^pjk - 2^{p+q}k = 2ijk.$$
This is a contradiction since the left-hand side is a multiple of 4 
whereas the right-hand side is not. 
\end{proof}

\paragraph{Acknowledgements}
The authors thank Lo\"{\i}c Chaumont who aroused 
our interest in this topic, and the referee for a careful 
reading and for a simplified proof of lemma~\ref{mixture}.

\begin{flushleft}
Jean Brossard and Christophe Leuridan\\
Institut Fourier, Universit\'e Joseph Fourier et CNRS\\
BP 74, 38$\,$402 Saint-Martin-d'H\`eres Cedex, France\\
\textit{jean.brossard@ujf-grenoble.fr},
\textit{christophe.leuridan@ujf-grenoble.fr}
\end{flushleft}

\end{document}